\documentclass{amsart}
\usepackage{amsmath,amssymb}

\numberwithin{equation}{section}

\newcommand\KK{{{K}}}

\newcommand{\RR} {\mathbb R}
\newcommand{\CC} {\mathbb C}
\newcommand{\ZZ} {\mathbb Z}
\newcommand{\NN} {\mathbb N}

\newcommand{\kk} {\mathbb K}

\newcommand{\Cal} {\mathcal}
\newcommand{\la} {\langle}
\newcommand{\ra} {\rangle}

\newcommand{\beq} {\begin{equation}}
\newcommand{\eeq} {\end{equation}}

\newtheorem{theorem}{Theorem}[section]
\newtheorem{remark}[theorem]{ Remark}

\newtheorem{corollary}[theorem]{Corollary}

\newtheorem{proposition}[theorem]{Proposition}

\newtheorem{lemma}[theorem]{Lemma}

\newtheorem{definition}[theorem]{Definition}

\begin{document}
\title[Structure spaces of rings of analytic functions ]{ Discrete $z$-filters and rings of analytic functions}
\author{Bedanta Bose and Mayukh Mukherjee}

\address{St. Xavier's College\\ 30, Mother Teresa Sarani\\ Kolkata - 700099\\ India}
\email{ana\_bedanta@yahoo.com}

\address{Max Planck Institute for Mathematics\\ Vivatsgasse 7\\ 53111 Bonn,
\\ Germany}

\email{mukherjee@mpim-bonn.mpg.de}
\begin{abstract} Consider rings of single variable real analytic or complex entire functions, denoted by $\kk\la z\ra$. 
	We study ``discrete $z$-filters'' on $\kk$ and their connections with the space of maximal ideals of $\kk\la z\ra$, which we characterize as a compact $T_1$ space $\theta \kk$ of discrete $z$-ultrafilters on $\kk$. We show that $\theta\kk$ is a bijective continuous image of 
$\beta \kk \setminus Q(\kk)$, where $Q(\kk)$ is the set of far points of $\beta\kk$. $\theta \kk$ turns out to be the Wallman compactification of the canonically embedded image of $\kk$ inside $\theta\kk$. Using our characterization of $\theta\kk$, we derive a Gelfand-Kolmogorov characterization of maximal ideals of $\kk\la z\ra$ and show that the Krull dimension of $\kk\la z\ra$ is at least $c$. We also establish the existence of a chain of prime $z$-filters on $\kk$ consisting of at least $2^c$ many elements. 
\end{abstract}
\maketitle
\section{Introduction}
For an arbitrary topological space $X$, let $C^*(X)$ be the ring of bounded continuous functions on $X$.  One has a canonical map $\mathcal{M}: X \rightarrow \mathfrak{M}(C^*(X))$, the space of all maximal ideals of $C^*(X)$ by taking $x$ to the maximal ideal $\mathfrak{m}_x$ of all functions vanishing at $x$.  This map is continuous when the codomain is given the so-called Stone topology, which we describe briefly:
for any $f \in C^*(X)$, define
\beq
M(f) = \{M \in \mathfrak{M}(C^*(X)) : f \in M\}.
\eeq
These sets $M(f)$ define a base for the closed sets on $\mathfrak{M}$. The following facts are well-known (see ~\cite{GJ}):
\begin{enumerate}
\item $\mathfrak{M}(C^*(X))$ is compact in the Stone topology.
\item For a Tychonoff $X, \mathfrak{M}(C^*(X))$ is (homeomorphic to) the well-known Stone-\v{C}ech compactification of $X$.
\end{enumerate}

As long we are on $\RR^n$, where the notion of smoothness 
makes sense, one can prove that the structure space of bounded real-valued smooth functions on $\RR^n$ is also $\beta\RR^n$. One must, however, be careful with further extensions: for instance, by a result of Grauert ( ~\cite{Gr}), we know that real analytic functions are dense in continuous functions for the Whitney $C^0$-topology for any paracompact real analytic manifold. However, the zero sets of real analytic functions have zero interior, meaning that the ring of real analytic functions has no divisor of zero. Consequently, the structure space of the ring of real analytic functions on $\RR^n$ will not generate the Stone-\v{C}ech compactification of $\RR^n$; in fact, it will not even be Hausdorff.

In this paper, one of our principal aims is to give a unified treatment of the structure spaces of the following rings: 
\begin{enumerate}
\item Ring of single variable complex entire functions 
denoted by $\CC\la z\ra$. 
\item Ring of convergent power series in one complex variable with strictly real coefficients, denoted by $\CC\la x\ra$\footnote{The proper notation for this ring should have been $\RR\la z\ra$ instead. The reason behind this notational idiosyncrasy will become clear later on. Just for the sake of clarity, by $\CC\langle x\rangle$ we are referring to power series of the form $\Sigma_{n = 0}^\infty a_nz^n$, where $a_n \in \RR, z \in \CC$.}.
\item Ring of single variable real analytic functions, which are given locally by convergent power series in one real variable with real coefficients, denoted by $\RR\la x\ra$. 
\end{enumerate}

In general, unless we need to specialize, we will refer to all the three collectively as $\kk\la z\ra$, where $\kk$ stands for the field $\RR$ or $\CC$, as the case may be. Also, observe that $\RR\la x\ra$ is different from the ring of real entire functions, which are functions of the form $\sum_{0}^{\infty}a_nx^n$, $a_i \in \RR$, where $\lim_{n \to \infty}|a_n|^{1/n} = 0$.

To our knowledge, not too much is known regarding the structure space of the ring of entire functions of a single complex variable. Some of the available literature also concentrates on bounded holomorphic functions defined on the open unit disc. Investigations into the variants of such spaces started quite early, some of the relatively older literature are ~\cite{A}, ~\cite{H}, ~\cite{He}, ~\cite{Kr}, ~\cite{R}, ~\cite{W}, to name a few. However, there is also evidence of contemporary interest into such questions, as can be gathered from ~\cite{G}, ~\cite{GH}, ~\cite{P}, and references therein. We observe right off the start that the structure space in our case is not going to be Hausdorff, as $\kk\la z\ra$ does not contain any divisors of zero, and hence is an integral domain. 
The novelty of this paper is then to look at rings of unbounded functions producing non-Hausdorff compactifications, in contrast with much of the earlier literature.

En route our investigation of structure spaces of $\kk\la z\ra$, we are led to examine zero sets of analytic functions, and the so-called discrete $z$-filters formed they form; roughly speaking, these are filters consisting of closed discrete sets of $\kk$. Dual to such filters is the notion of so-called discrete $z$-ideals of $\kk\la z\ra$ (see Definition \ref{dzi}). We are led to establishing several properties of discrete prime $z$-filters  and discrete $z$-ultrafilters. Not surprisingly, they turn out to be dual to prime and maximal $z$-ideals respectively. It also turns out to be a special property of $\kk\la z\ra$ that all ideals are discrete $z$-ideals. 


Our main results in Section \ref{sec2} are the following. In Proposition \ref{peum}, we establish that every prime ideal extends uniquely to a maximal ideal in $\kk\la z\ra$, which finally falls in line with the established paradigm, namely, establishing a correspondence between each maximal ideal of $\kk\la z\ra$ and a ``discrete $z$-ultrafilter'' on $\kk$. In Proposition \ref{peum1}, we demonstrate that every fixed prime ideal in $\kk\la z\ra$ is a maximal ideal. More importantly, in Theorem \ref{2^{c} prime}, we are able to establish the existence of a chain of prime $z$-filters on $\kk$ consisting of at least $2^c$ elements. This extends results in ~\cite{M} and ~\cite{P}, see the end of Section \ref{sec2} for a discussion. The proof makes use of Henriksen's well-known result that the Krull dimension of $\kk\la z\ra$ is at least $2^c$. 

In Section \ref{sec3}, we define $\theta \kk$ (see Definition \ref{thetake}), the space of discrete $z$-ultrafilters on $\kk$, which is our candidate for the structure space of $\kk\la z\ra$; we demonstrate a homeomorphism between $\theta\kk$ and $\mathfrak{M}(\kk\la z\ra)$ in Theorem \ref{hwtk}
. Since $\theta\kk$ consists of discrete $z$-ultrafilters, heuristically speaking, it seems reasonable that it will not ``reach'' any far point of $\beta \kk$. We make this heuristic precise in Theorem \ref{bci}, where we explicitly construct a continuous bijective map $\Psi$ from $\beta\kk \setminus Q(\kk)$ to $\theta\kk$. 

Now, let $\eta_\kk : \kk \to \theta\kk$ denote the canonical embedding (see Subsection \ref{3.2}). 
In Theorem \ref{max}, we derive a characterization of maximal ideals of $\kk\la z\ra$ in the fashion of the Gelfand-Kolmogorov theorem. 
A non-trivial application of this characterization is Theorem \ref{krull}, which gives a method of somewhat explicitly constructing a chain of prime ideals of $\kk\la z\ra$ from a chain of neighbourhoods of a point $p \in \theta\kk \setminus \eta_\kk(\kk)$. As a corollary, this gives that the Krull dimension of $\kk\la z\ra$ is at least $c$ (Corollary \ref{krull1}). 
It seems tempting to speculate that the proof of Corollary \ref{krull1} might be extended to give a different proof of Henriksen's well-known result about the Krull dimension of $\kk\la z\ra$ (this will be independent of Theorem \ref{2^{c} prime}). 

In Theorem \ref{ska}, we prove that $\theta\kk$ is the Wallman compactification of $\eta_\kk(\kk)$. 
It is clear that $\CC\la z\ra$ and $\RR\la x\ra$ are trivially non-isomorphic ($-1$ belongs to both rings, has a square root in $\CC\la z\ra$, and not in $\RR\la x\ra$). But, curiously enough, we are able to show via rather elementary arguments that $\eta_\RR(\RR)$ and $\eta_\CC(\CC)$ are homeomorphic, which implies that $\theta \RR$ and $\theta \CC$ are homeomorphic. In the context of $\beta X$, one can produce examples of non-isomorphic continuous function rings with homeomorphic structure spaces; for example, for a Tychonoff space $X$, any ring between $C^*(X)$ and $C(X)$ has structure space $\beta X$. 
Here we encounter a similar example for Wallman compactifications.


{\bf Note:} In this paper, some of the results (notably, the existence of far points of $\beta X$ and the proof of Theorem \ref{2^{c} prime}) depend on the Continuum Hypothesis.

\subsection{Organization of the paper} In Subsection \ref{ssec21}, we fix our notations, and collect together some background material and also the main technical lemmas about discrete $z$-ideals and filters. In Subsection \ref{ssec22}, we prove our main results for Section \ref{sec2}, which are Propositions \ref{peum} and \ref{peum1} and Theorem \ref{2^{c} prime}. In Subsection \ref{ssec31}, we make a brief discussion of far points of $\beta \kk$ and their connections to discrete $z$-filters. Our main results for Section 3 are contained in Subsections \ref{3.1} and \ref{3.2}. There are respectively Theorems \ref{hwtk}, \ref{bci}, \ref{max}, \ref{krull}, Corollary \ref{krull1}, and Theorem \ref{ska}. We finish the paper with an Appendix containing some assorted results about $\theta\kk$ and the map $\Psi$. Though not indispensable for our main results, we believe them to be of independent interest.  
\section{(Discrete) $z$-ideals and (discrete) $z$-filters}\label{sec2}
\subsection{Notations, definitions and technical lemmas}\label{ssec21}
Before we begin, let us standardize our notations for the rest of the paper (we make a complete list, at the risk of a little repetition):
\begin{enumerate}
\item  $C^*(X)$ : bounded continuous functions on $X$.
\item $\mathfrak{M}$ : space of maximal ideals.
\item $\kk\la z\ra$ : convergent power series in the variable $z$ with coefficients belonging to $\kk$.
\item $\Cal{D}$ : discrete set.
\item $Z_f$ : zero set of the function $f$.
\item $\beta X$ : Stone-\v{C}ech compactification of $X$.
\item $c$ : cardinality of $\RR$.
\item $\mathfrak{F}_p$ : the $p$ in the {\em subscript} denotes a filter $\mathfrak{F}$ fixed or supported at $p$.
\item $\text{cl}_X(Y)$ : the closure of $Y$ in $X$.
\item $I$ : a proper ideal of $\kk\la z\ra$.
\item $\la \Cal{F}\ra$: $z$-filter generated by $\Cal{F}$.
\item $\la \Cal{F}\ra^\uparrow$ : $z$-ultrafilter containing $\la \Cal{F}\ra$.
\end{enumerate}
At the outset, we note the following well-known facts which will be crucially used throughout the paper. Firstly, we have the following:
\begin{lemma} \label{impa}
Given any discrete set (with multiplicities) $\Cal{D}$ of $\kk$ ($\kk = \RR$ or $\CC$), by the Weierstrass factorization theorem, we can find an $f \in \CC\langle z\rangle$ and $\RR\langle x\rangle$ respectively such that $\Cal{D} = Z_f$. If $D$ is symmetric about the $x$-axis, then we can find $f \in \CC\langle x\rangle$ such that $\Cal{D} = Z_f$.
\end{lemma}
Secondly, we prove that $\kk\la z\ra$ is a gcd domain. The cases $\CC\la z\ra$ and $\CC\la x\ra$ are already contained in the proof of Theorem 9 in ~\cite{H}. Here we modify Helmer's method to prove the corresponding statement for $\RR\la x\ra$:
\begin{lemma}\label{impb}
Consider a non-trivial proper ideal $I$ of $\RR\la x\ra$. If $f, g \in I$, then the g.c.d. of $f$ and $g$ is also in $I$.
\end{lemma}
\begin{proof}
We see that it suffices to prove that the g.c.d. $d$ of $f$ and $g$ can be expressed as $d = \alpha f + \beta g$, where $\alpha, \beta \in \RR\la x\ra$. By considering $\frac{f}{d}$ and $\frac{g}{d}$, it suffices to prove that if $f, g \in \RR\la x\ra$ are relatively prime (that is, g.c.d.$(f, g) = 1$), then there are $\alpha, \beta \in \RR\la x\ra$ such that, 
\[
1 = \alpha f + \beta g.\]

Since $f$ and $g$ cannot vanish at $0$ simultaneously (otherwise $h$ given by $h(x) = x$ would be a common factor of $f$ and $g$), assume without loss of generality that $g (0) \neq 0$. Also, let $g$ have the zeros $p_1, p_2,...$ with multiplicities $m_1, m_2,...$. Letting $p$ stand for any one of the $p_i$'s occurring with multiplicity $m$, expand $fg$ in a power series locally around $p$:
\[
f(x)g(x) = w_m(x - p)^m + w_{m + 1}(x - p)^{m + 1} + ....
\]
with $w_m \neq 0$. 

We wish to find  a real-valued ``meromorphic'' function $M(x)$ of one real variable which has a single real variable Laurent series expansion such that its only ``poles'' are of order $m_i$ at $p_i$ with the singular part at the generic pole $p$ being equal to \[
c_m(x - p)^{-m} + ...  + c_1(x - p)^{-1},\]
where $c_i \in \RR$. Since $g$ is real analytic, around any point $x_0 \in \RR$, it has a power series expansion $g(x) = \sum^\infty_{n = 0} \frac{g^{(n)}(x_0)}{n!}(x - x_0)^n$. Observe that the function $g(z) := \sum^\infty_{n = 0} \frac{g^{(n)}(x_0)}{n!}(z - x_0)^n$ is a local holomorphic extension of $g$, defined in an open domain $\Omega$ containing $\RR$. 
Now, choose in $\Omega$ (by the Mittag-Leffler theorem), a meromorphic function $M(z)$. This means, from the expansion $f(z)g(z) = w_m(z - p)^m + w_{m + 1}(z - p)^{m + 1} + ....$ near $p$ (note that the coefficients $w_i$ are real-valued), we determine real numbers $c_1, c_2,..., c_m$ in the following way:
\begin{align*}
w_mc_m & = 1,\\
w_mc_{m - 1} + w_{m + 1}c_m & = 0,\\
...\text{   }.....\text{   }.....& \text{  }.....\\
w_mc_1 + w_{m + 1}c_2 + .... + w_{2m - 1}c_m & = 0.
\end{align*} 
Now, define $M'(x) = \text{Re } M(x)$, where $M(x)$ represents the restriction of the function $M(z)$ to $\RR$. This  will serve as the meromorphic function we were looking for, and $g$ will be a divisor of $1 - \alpha f$, where $\alpha (x) = M'(x)g(x)$. This concludes the proof.  
\end{proof}
\subsubsection{Preliminary lemmas}
Now we define the concept of discrete $z$-filters, arguably the main technical tool for our investigation. Denote by $\Cal{D}(\kk)$ the collection of all closed discrete sets inside $\kk$ and also containing the special member $\kk$.
\begin{definition}\label{dzf} A collection $\mathcal{F}$ of subsets of $\Cal{D}(\kk)$ is called a discrete $z$-filter on $\kk$ if 
\begin{enumerate}
\item $\emptyset\notin \mathcal{F}$.
\item if $D_{1}, D_{2} \in \mathcal{F}$ then $D_{1}\cap D_{2} \in \mathcal{F}$.
\item if $D_{1}\in \mathcal{F}$ and $D\in \Cal{D}(\kk)$ such that $D_{1}\subseteq D$ then $D\in \mathcal{F}$.
\end{enumerate}
\end{definition}

Now we quickly write down the following 
\begin{lemma}
(a) If $I$ is an ideal of $\kk\la z\ra$ then $Z[I]=\{Z_f : f\in I\}$ is a discrete $z$-filter on $\mathbb{K}$.\newline
(b) If $\mathcal{F}$ is a discrete $z$-filter on $\mathbb{K}$, then $Z^{-1}[\mathcal{F}]=\{f\in \kk\langle z\rangle : Z_f\in \mathcal{F}\}$ is an ideal of $\kk\la z\ra$.
\end{lemma}
\begin{proof} (a) It is clear that $\emptyset \notin Z[I]$. Now for $f,g \in I$, $Z_f$ and $ Z_g$ are closed discrete subsets of $\kk$ and therefore $Z_f\cap Z_g$ is also closed and discrete subset of $\kk$. Now by Lemmas \ref{impa} and \ref{impb}, $Z_f\cap Z_g = Z_d$, where $d$ is the greatest common divisor of $f$ and $g$. Hence $d\in I$ (by Lemma \ref{impb}), and therefore $Z_f\cap Z_g\in Z[I]$. Again, let $Z \in \Cal{D}(\kk)$ and $f\in I$ such that $Z_f \subset Z$. Then, there exists $g\in \kk\la z\ra$ such that $Z=Z_g$ and then $Z=Z_f\cup Z_g = Z_{fg} \in Z[I]$. \newline
(b)  Let $f, g \in Z^{-1}[\mathcal{F}]$ then $Z_f,Z_g\in \mathcal{F} \Rightarrow Z_f\cap Z_g\subseteq Z_{f+g}\in \mathcal{F}$ 
which finally implies that $f+g \in Z^{-1}[\mathcal{F}]$.\newline
Let $f\in Z^{-1}[\mathcal{F}]$ and $ g \in \kk\la z\ra$ then $Z_f\subseteq Z_f\cup Z_g=Z_{fg} \in \mathcal{F} \Rightarrow fg\in Z^{-1}[\mathcal{F}]$. 
\end{proof}
Now, we see that a discrete $z$-filter $\Cal{F}$ on $\kk$ can be extended to a maximal discrete $z$-filter by Zorn's lemma, which we call a discrete $z$-ultrafilter on $\kk$. The following lemma is automatic:
\begin{lemma}\label{lmmbij}\label{bijct} If $M$ is a maximal ideal of $\kk\la z\ra$ then $Z[M]$ is a discrete $z$-ultrafilter on $\kk$. Conversely, if $\mathcal{F}$ is a discrete $z$-ultrafilter on $\kk$ then $Z^{-1}[\mathcal{F}]$ is a maximal ideal of $\kk\la z\ra$.
\end{lemma}
\begin{proof} Since $Z[M]$ is a discrete $z$-filter on $X$, by Zorn's lemma, $Z[M]$ is contained in a discrete $z$-ultrafilter on $\Cal{F}$. Therefore $Z^{-1}[Z[M]] \subset Z^{-1}[\Cal{F}]$ which imply that $M \subset Z^{-1}[\Cal{F}]$. But $M$ is a maximal ideal so $M = Z^{-1}[\Cal{F}]$ and hence $Z[M] = \Cal{F}$. The converse can be checked similarly. 
\end{proof}

\begin{definition} \label{dzi} An ideal of $\kk\la z\ra$ is called a discrete $z$-ideal if $Z^{-1}[Z[I]]=I$.
\end{definition}
The following is a rather special property of $\kk\la z\ra$:
\begin{lemma}\label{dhopas}
	Every ideal of $\kk\la z\ra$ is a discrete $z$-ideal.
\end{lemma}
\begin{proof}
	Let $I$ be an ideal of $\kk\la z\ra$. The $Z[I]$ is a discrete $z$-filter. Let $Z\in Z[I]$. Then, consider $Z_{1}=Z\cup \{p\}$ and $Z_{2}=Z\cup\{q\}, \quad(p\neq q)$, both discrete in $\kk$, and $Z_{1}\cap Z_{2}= Z=Z_h$ for some $h\in \kk\la z\ra$. Since $Z_{1}, Z_{2}\in Z[I]$, there exist $f,g\in I$ such that $Z_{1}=Z_f$ and $Z_{2}=Z_g$. Then from Lemma \ref{impb}, $h$ is the gcd of $f,g$ and hence it belongs to $I$. In other words we have proved that $\{h\in \kk\la z\ra : Z_h = Z\}\subseteq I$.
\end{proof}
 
 As rather trivial applications of Lemma \ref{dhopas}, we prove that all non trivial ideals in $\RR\la x\ra$ are formally real and all non trivial ideals of $\CC\la z\ra$ are formally complex (for maximal ideal case these results were proved using different methods in ~\cite{GH}). That is,
 \begin{corollary}
 	If $I$ is a non trivial ideal of $\RR\la x\ra$ and $f^2 + g^2 \in I$, then $f, g \in I$.
 \end{corollary}
 \begin{proof}
 	$f^2 + g^2 \in I$ means that $Z_{f^2 + g^2}\in Z[I]$. But $Z_{f^2 + g^2} = Z_f\cap Z_g$, giving that $Z_f, Z_g\in Z[I]$. Hence Lemma \ref{dhopas} implies $f, g \in I$.
 \end{proof}
 \begin{corollary}
 	If $I$ is a non trivial ideal of $\CC\la z\ra$, then $r^2+1\in I$ for some $r\in \CC\la z\ra$.
 \end{corollary}
 \begin{proof}
 	Since $I$ is non trivial so there exists $f\in I$ such that $Z_f\in Z[I]$ and $Z_f$ is non-empty. Let $g=(f+i)^2+1 \in \CC\la z\ra$. Then $Z_f\subseteq Z_g$ implies that $Z_g\in Z[I]$. Hence $g\in I$ follows from Lemma \ref{dhopas}. 
 \end{proof}
 
\begin{definition}\label{dpzf} A discrete $z$-filter $\mathcal{F}$ is called a discrete prime $z$-filter on $\kk$ if for any two $Z_{1}, Z_{2} \in \Cal{D}(\kk)$, $Z_{1}\cup Z_{2} \in \mathcal{F} \text{  implies  } ~Z_{1}\in \mathcal{F}~ \mbox{or}~ Z_{2} \in \mathcal{F}$.
\end{definition}
It follows easily that every discrete $z$-ultrafilter on $\kk$ is a discrete prime $z$-filter.

\begin{lemma} If $I$ is a prime ideal of $\mathbb{K}\langle z\rangle$ then $Z[I]$ is a discrete prime $z$-filter and conversely.
\end{lemma}
\begin{proof} Let $Z_{1},Z_{2}\in \Cal{D}(\mathbb{K})$ such that $Z_{1}\cup Z_{2}\in Z[I]$. Let $Z_{1}=Z(f_{1})$ and $Z_{2}=Z(f_{2})$. Therefore $Z(f_{1})\cup Z(f_{2})\in Z[I]$. Since $I$ is a discrete $z$-ideal so $f_{1}f_{2}\in I$ imply that $f_{1}\in I \mbox{  or  } f_{2}\in I$ and therefore $Z_{1}\in Z[I] \mbox{   or   } Z_{2}\in Z[I]$. Conversely, let $\mathcal{F}$ be a discrete prime $z$-filter on $\mathbb{K}$. Let $f_{1}f_{2}\in Z^{-1}[\mathcal{F}]$ and so $Z(f_{1}f_{2})\subseteq \mathcal{F}$, i.e., $Z(f_{1})\cup Z(f_{2})\in \mathcal{F}$ and consequently either $Z(f_{1})$ or $Z(f_{2})\in \mathcal{F}$,i.e., $f_{1}\in Z^{-1}[\mathcal{F}]$ or $f_{2}\in Z^{-1}[\mathcal{F}]$.
\end{proof}
Any family $\mathcal{F}$ of discrete closed sets in $\kk$ with a finite intersection property is contained in a $z$-filter of zero sets in $\kk$. The smallest such $z$-filter is said to be generated by $\mathcal{F}$. It $\mathcal{F}$ is also a discrete $z$-filter on $\kk$, then it is closed under finite intersection, therefore it also forms a base for the $z$-filter which it generates. 
The following lemma can be easily checked using the definition of $z$-filters.
\begin{lemma}\label{cntd} If $\mathcal{F}$ and $\mathcal{F'}$ are $z$-filters on $\kk$, then $\mathcal{F}\subsetneq \mathcal{F'} \Rightarrow \langle\mathcal{F}\rangle\subsetneq \langle\mathcal{F'}\rangle$.
\end{lemma}
\subsection{Main theorems}\label{apps}\label{ssec22}
Now we prove our main results for Section \ref{sec2}. 
Our first result shows a rather important connection between prime $z$-filters and discrete prime $z$-filters on $\kk$.
\begin{proposition}\label{hil gaya} If $\mathcal{F}$ is a discrete prime $z$-filter on $\kk$ then $\langle\mathcal{F}\rangle$ is a prime $z$-filter on $\kk$. Conversely, if $\mathcal{F}'$ is a prime $z$-filter on $\kk$ which contains discrete closed sets of $\kk$ then there exists a discrete prime $z$-filter $\mathcal{F}$ such that $\langle\mathcal{F}\rangle=\mathcal{F}'$.
\end{proposition}
\begin{proof} First of all, $\langle\mathcal{F}\rangle$ is a $z$-filter on $\kk$. As regards primality, let $Z_{1}, Z_{2} \in Z[\kk]$ (the family of all zero sets of continuous functions in $\kk$) such that $Z_{1}\cup Z_{2} \in \langle\mathcal{F}\rangle$. This means that there exists $Z\in \mathcal{F}$ such that $Z\subseteq Z_{1}\cup Z_{2}\in \langle\mathcal{F}\rangle$ (since $\mathcal{F}$ is a base for $\langle\mathcal{F}\rangle$). This means that $(Z\cap Z_{1})\cup (Z\cap Z_{2})=Z\cap (Z_{1}\cup Z_{2})=Z\in \mathcal{F}$. Since each of $Z\cap Z_{1}$ and $Z\cap Z_{2}$ is discrete and $\mathcal{F}$ is a discrete prime $z$-filter on $\kk$, therefore $Z\cap Z_{1}$ or $Z\cap Z_{2}$ belongs to $\mathcal{F}$ and it implies that $Z_{1}\in \langle\mathcal{F}\rangle$ or $Z_{2}\in \langle\mathcal{F}\rangle$. For the converse part one can consider the set $\mathcal{F}$ consisting of discrete closed sets in $\mathcal{F}'$. It is now easy to check that $\mathcal{F}$ is a discrete prime $z$-filter and $\langle\mathcal{F}\rangle=\mathcal{F}'$. 
\end{proof}

Seeing that every prime $z$-filter is contained in a unique $z$-ultrafilter on $\kk$ (see ~\cite{GJ}, Section 2.13), we can conclude
\begin{corollary}\label{gnrt} Every discrete prime $z$-filter of discrete closed sets of $\kk$ generates a unique $z$-ultrafilter of zero sets of $\kk$. In particular therefore every discrete $z$-ultrafilter $\Cal{F}$ of $\kk$ generates a unique $z$-ultrafilter $\langle\Cal{F}\rangle^\uparrow$ of zero sets of $\kk$.
\end{corollary}
Now, we prove two important and rather special properties of prime ideals of $\kk\la z\ra$.
\begin{proposition}\label{peum}
	Every prime ideal of $\kk\la z\ra$ is contained in a unique maximal ideal of $\kk\la z\ra$.
\end{proposition}
\begin{proof} Let $P$ be a ideal of $\kk\la z\ra$ and suppose that $P$ is contained in two distinct maximal ideals $M$ and $M'$ of $\kk\la z\ra$. Then by Lemma \ref{cntd} the discrete prime $z$-filter $Z[P]$ is contained in both $Z[M]$ and $Z[M']$ and therefore $\langle Z[P]\rangle$ is contained in both $\langle Z[M]\rangle$ and $\langle Z[M']\rangle$. Since $M$ and $M'$ are distinct, therefore by Lemma \ref{bijct} $Z[M]$ and $Z[M']$ are distinct and hence $\langle Z[M]\rangle^\uparrow$ and $\langle Z[M']\rangle^\uparrow$ are distinct. But $Z[P]$ is contained in both $\langle Z[M]\rangle^\uparrow$ and $\langle Z[M']\rangle^\uparrow$. This contradicts the first part of Lemma \ref{gnrt}. Hence the claim follows.
\end{proof}
\begin{proposition}\label{peum1}
	Every fixed prime ideal is maximal.
\end{proposition}
\begin{proof}
	Let $P$ be a fixed prime ideal of $\kk\la z\ra$. Then $Z[P]$ is a discrete prime $z$-filter. Since $Z[P]$ is fixed, $\bigcap Z[P]=\{p\}, p\in \kk$. $Z\in Z[P]\Rightarrow p\in Z$. 
	Now, $Z = (Z\setminus \{p\})\cup \{p\} \in Z[P] \Rightarrow\{p\}\in Z[P]$ (since $Z[P]$ is prime and $Z\setminus \{p\}\notin Z[P]$). Now the upset property of filter ensure that $Z[P]$ is a discrete $z$-ultrafilter, which gives that $Z[P]$ is a discrete $z$-ultrafilter. Now from  Lemma \ref{bijct} it follows that $P$ is maximal.
\end{proof}


Prime $z$-ideals have interesting connections with the topology of the underlying space; for example, in the case of a completely regular Hausdorff space $X$, prime $z$-ideals are related to convergence problems in the Stone-\v{C}ech compactification $\beta X$. Observe that Proposition \ref{hil gaya} essentially tells us that the structure of discrete prime $z$-filters of $\kk$ and prime $z$-filters of $\kk$ which contain discrete zero sets are same as partially ordered sets. Golasinski and Henriksen proved (see ~\cite{GH}) that the Krull dimension of $\kk\langle z \rangle$ is at least $2^{c}$ under Continnum Hypothesis, i.e., there is a chain of prime ideals in $\kk\langle z \rangle$ consisting of $2^{c}$ many elements. Then from Lemma \ref{cntd}, Proposition \ref{hil gaya} and Lemma \ref{dhopas}, we can conclude the following:
\begin{theorem}\label{2^{c} prime}
There is a chain of prime $z$-filters on $\kk$ consisting of at least $2^{c}$ many prime $z$-filters and each has a base consisting of discrete closed sets.
\end{theorem} 

To put Theorem \ref{2^{c} prime} in proper perspective, let us recall that in general, for a Tychonoff space $X$, there are examples of prime maximal ideals which contain no other prime ideals except themselves (for example, refer to ~\cite{M}, pp 157). Also, there are examples of spaces where only prime $z$-ideals are $M^{p}$ and $O^{p}$ but there are $2^{c}$ many prime ideals in between them (see ~\cite{GJ}, pp 200). ~\cite{Ko} seems to be one of the first investigations into the structure of prime $z$-filters. In ~\cite{M}, it is shown (Theorems 13.3, 14.1) that a chain of prime $z$-ideals in $C^*(\RR)$ is countably infinite. Recently, ~\cite{P} has shown that $\RR$ has a chain of prime $z$-filters of cardinality $c$, which seems to be the optimal result known till now.


\section{Topological properties of $\mathfrak{M}(\kk\la z\ra)$}\label{sec3}
\subsection{Far points and discrete $z$-filters}\label{ssec31}
From the discussion in the last section, we see that every discrete prime $z$-filter extends to a unique $z$-ultrafilter on $\kk$. But there are $z$-ultrafilters on $\kk$ which do not contain any discrete subset of $\kk$. It follows from the following result:
\begin{theorem} (~\cite{FG})\label{12} If $X$ is a non-pseudocompact space which contains no more than $\aleph_{1}$ many dense open sets, then there is a free $z$-ultrafilter on $X$ no member of which is nowhere dense. 
\end{theorem}

Our space $\kk$ satisfies all properties of the above result and hence $\beta \kk$ has many far points. Obviously no discrete prime $z$-filter on $\kk$ can be extended to the kind of $z$-ultrafilters on $\kk$ mentioned in Theorem \ref{12}. But if we restrict our collection of $z$-ultrafilters on $\kk$ from $\beta \kk$ to $\beta \kk\setminus Q(\kk)$, where $Q(\kk)$ is the collection of all far points of $\beta \kk$, then for this restricted space the following proposition ensures that every $z$-ultrafilter on $\kk$ can be achieved as an extension of some discrete $z$-ultrafilter. 
\begin{proposition}\label{propi} For every $z$-ultrafilter $\mathcal{U}^{p}$ on $\kk$ where $p\in \beta \kk\setminus Q(\kk)$ there exists a unique discrete $z$-ultrafilter $\mathcal{F}$ on $\kk$ such that $\langle\mathcal{F}\rangle$ extends to $\mathcal{U}^{p}$.
\end{proposition}
\begin{proof} Let $\mathcal{F}=\{Z\in \mathcal{U}^{p} : Z\in \Cal{D}(\kk)\}$. Since $p\in \beta\mathbb{K}\setminus Q(\mathbb{K})$, so $A$ is non-empty. It follows immediately that $A$ is a discrete $z$-filter on $\mathbb{K}$. Let us take two discrete zero sets $Z_{1},Z_{2}\in \Cal{D}(\mathbb{K})$ such that $Z_{1}\cup Z_{2}\in \mathcal{F}\subseteq \mathcal{U}^{p}$. As $\mathcal{U}^{p}$ is a prime $z$-filter it automatically shows that either $Z_{1}\in \mathcal{F}$ or $Z_{2}\in \mathcal{F}$. Hence $\mathcal{F}$ becomes a prime discrete $z$-filter on $\kk$. Consequently from Proposition \ref{gnrt} we can conclude that $\mathcal{F}$ is contained in a unique $z$-ultrafilter on $\kk$, i.e., $\mathcal{U}^{p}$ in this particular case. Now we intend to show that $\mathcal{F}$ is a discrete $z$-ultrafilter on $\kk$. Suppose $\mathcal{F}'$ is a discrete $z$-ultrafilter on $\kk$ such that $\mathcal{F}\subseteq \mathcal{F}'$. Then Proposition \ref{gnrt} implies also that $\mathcal{F}'$ extends to a unique $z$-ultrafilter on $\kk$ and obviously in this case that particular $z$-ultrafilter will be $\mathcal{U}^{p}$. Then from the construction of $\mathcal{F}$ it follows that $\mathcal{F}'\subseteq \mathcal{F}$. The uniqueness of $\mathcal{F}$ also follows from the construction of $\mathcal{F}$. 
\end{proof}

Proposition \ref{propi} is true also for prime $z$-filters on $\kk$. Since each prime $z$-filter is contained in a unique $z$-ultrafilter so Theorem \ref{2^{c} prime} can be modified as the following.

\begin{theorem}
If $p\in \beta\kk\setminus Q(\kk)$ then the $z$-ultrafilter $\mathcal{U}^{p}$ of $\kk$ contains a chain of prime $z$-filters on $\kk$ consisting of at least $2^{c}$ many elements.
\end{theorem}

\subsection{The relation between $\beta\kk$ and $\theta\kk$}\label{3.1} 

Recall that there is a one-one correspondence between the points of $\beta\kk$ and the $z$-ultrafilters on $\kk$, each $z$-ultrafilter converging to its corresponding point. Keeping this correspondence in mind, we define a map from the structure space $\mathfrak{M}(\kk\la z\ra)$ (henceforth denoted simply by $\mathfrak{M}_\kk$) to $\beta \kk\setminus Q(\kk)$ by 
$$\Phi: \mathfrak{M}_\kk \longrightarrow \beta\mathbb{K}\setminus Q(\mathbb{K}),~~ \Phi(M)=\langle Z[M]\rangle^\uparrow, $$
where $\langle Z[M]\rangle^\uparrow$ denotes the unique $z$-ultrafilter on $\kk$ to which the discrete $z$-ultrafilter $Z[M]$ can be extended. It is quite clear that this map is an injective map. We quickly write down some of its properties:

Surjectivity of $\Phi$: It follows immediately from Proposition \ref{propi}.

Closedness of $\Phi$: A basic closed set of $\mathfrak{M}_\kk$ is of the form $M(f)=\{M\in \mathfrak{M}_{\kk} : f\in M\}$, $f\in \kk\la z\ra$. We have, 
\begin{align*}
\Phi(M(f))&=\{\Phi(M) : f\in M\} = \{\langle Z[M]\rangle^\uparrow : f\in M\}\\
&=\{\langle Z[M]\rangle^\uparrow : Z_f\in Z[M]\} =\{\Cal{A}^{p}\in \beta\kk : Z_f\in Z[M]~\&~\langle Z[M]\rangle^\uparrow=\Cal{A}^{p}\}\\
&=\{\Cal{A}^{p}\in \beta\kk : Z_f\in\Cal{A}^{p}, p\notin Q(\kk)\} =\{\Cal{A}^{p} : p \in \beta\kk \setminus Q(\kk), Z_f\in \Cal{A}^{p}\}
\end{align*} 

which is a basic closed set in $\beta \kk \setminus Q(\kk)$. Therefore $\Phi$ is a closed map.
This gives us

\begin{theorem} \label{continuous image} There is a bijective continuous map from $\beta \kk\setminus Q(\kk)$ onto  $\mathfrak{M}_{\kk}$, the structure space of $\kk\la z\ra$.
\end{theorem}
\begin{definition}\label{thetake}
For each discrete $z$-ultrafilter on $\kk$, assign a point and make an indexed set for the set of all such discrete $z$-ultrafilers. We call this indexed set $\theta \kk$. For each point $p\in \theta \kk$, let $\Cal{D}^{p}$ denote the corresponding discrete $z$-ultrafilter on $\kk$. For a closed discrete subset $Z$ of $\kk$, denote $\overline{Z}=\{p\in \theta \kk : Z\in \Cal{D}^{p}\}$. We topologize $\theta \kk$ by considering all $\overline{Z}$ as the basic closed sets.
\end{definition}
\begin{lemma}\label{intp}
For distinct $Z_{1},Z_{2}\in \Cal{D}(\kk)$, $\overline{Z_{1}\cap Z_{2}} = \overline{Z_{1}} \cap \overline{Z_{2}}$.
\end{lemma}
\begin{proof}
One part is trivial, i.e., $\overline{Z_1}\cap \overline{Z_2} \subseteq \overline{Z_{1}} \cap \overline{Z_{2}}$. For the reverse part, $\overline{Z_1} \cap \overline{Z_{2}} =\{p\in \theta \kk : Z_{1},Z_{2}\in \Cal{D}^{p}\}\subseteq \{p\in \theta \kk : Z_{1}\cap Z_{2}\in \Cal{D}^{p}\}=\overline{Z_{1}\cap Z_{2}}$. 
\end{proof}
\begin{theorem}\label{ct1}
$\theta \kk$ is a compact $T_{1}$ space.
\end{theorem}
\begin{proof}
Consider two discrete $z$-ultrafilters $\Cal{D}^{p}$ and $\Cal{D}^{q}, p\neq q$. Since $\Cal{D}^{p}$ and $\Cal{D}^{q}$ are two distinct discrete $z$-ultrafilters there exists two distinct discrete closed sets $Z_{1}, Z_{2}$ such that $Z_{1}\in \Cal{D}^{p}$ and $Z_{1}\not\in \Cal{D}^{q}$ and $Z_{2}\in \Cal{D}^{q}$ and $Z_{2}\not\in \Cal{D}^{p}$ $\Rightarrow \Cal{D}^{p}\in \overline{Z}_{1}$, $\Cal{D}^{q}\not\in \overline{Z}_{1}$ and $\Cal{D}^{q}\in \overline{Z}_{2}$, $\Cal{D}^{p}\not\in \overline{Z}_{2}$.

To show compactness we need to show that a collection of basic closed sets with finite intersection property has non-empty intersection. Let $\{\overline{Z}_{\lambda}\}_{\lambda\in \Lambda}$ be a collection of basic closed sets in $\theta\kk$ with finite intersection property. Consider the collection $\mathcal{F}=\{Z_{\lambda} : \lambda\in \Lambda\}$. Then by Lemma \ref{intp}, $\mathcal{F}$ has finite intersection property. Therefore it can be extended to a discrete $z$-ultrafilter $\Cal{D}^{p}$. Then obviously $p\in \cap \overline{Z}_{\lambda}$. Therefore $\theta\mathbb{K}$ becomes compact.
\end{proof}
\begin{theorem}\label{hwtk} The structure space of $\kk\la z\ra$, that is, $\mathfrak{M}_{\kk}$, is homeomorphic with $\theta \kk$.
\end{theorem}
\begin{proof}
The map $\psi : \mathfrak{M}_\kk \longrightarrow \theta \kk$ given by $\psi(M)=Z[M]$ gives the homeomorphism. Bijectivity of $\psi$ is assured by the Lemma \ref{lmmbij}. Furthermore, $\psi$ exchanges the typical basic closed sets $M(f)$ and $\overline{Z_{f}}$ in the respective spaces.
\end{proof} 
Combining with Theorem $\ref{continuous image}$ we can conclude that
\begin{theorem} \label{bci}There is a bijective continuous map $\Psi : \beta\mathbb{K}\setminus Q(\mathbb{K}) \to \theta\mathbb{K}$.
\end{theorem}
\subsection{Embedding properties of $\kk$ in $\theta\kk$}\label{3.2}
Now we consider the natural embedding map $\eta_{\kk}:\kk\longrightarrow \theta\kk$ defined by $\eta_{\kk}(p)=\Cal{D}_{p}$, where $\Cal{D}_p$ is the discrete $z$-ultrafilter fixed at $p$, that is, all the members of $\Cal{D}_p$ contain the point $p$. This map turns out to be a continuous bijective map from $\kk$ onto $\eta_{\kk}(\kk)$-under the subspace topology of $\theta\kk$. But it can not be a homeomorphism as the basic closed subsets of $\eta_{\kk}(\kk)$ under the subspace topology of $\theta\kk$ are given by the $\eta_{\kk}$-image of the discrete closed subsets of $\kk$, which means that every closed subset of $\eta_{\kk}(\kk)$ is countable. It is quite obvious that $\eta_{\kk}(\kk)$ inherits the $T_{1}$ topology from $\theta\kk$. In our next result we prove that $\eta_{\kk}(\kk)$ is dense in $\theta\kk$, which allows us to interpret $\theta\kk$ as a kind of a compactification of $\kk$.
\begin{proposition}
$\eta_{\kk}(\kk)$ is dense in $\theta\kk$.
\end{proposition} 
\begin{proof}
Consider a basic open set $\theta\kk \setminus \overline{Z}$ of $\theta\kk$, where $Z$ is a discrete closed subset of $\kk$. Obviously $Z$ is not the whole of $\kk$, so consider a point $p\in \kk$ so that $p\notin Z$. Then $\eta_{\kk}(p)\notin \overline{Z}$ which implies that $\eta_{\kk}(p)\in \theta\kk\setminus \overline{Z}$.
\end{proof}
We need the following technical lemma:
\begin{lemma}\label{3.9}
(a) For each $Z\in \Cal{D}(\kk)$, $\overline{Z}\cap \eta_{\kk}(\kk)=\eta_{\kk}(Z)$.\newline
(b) 
For each $Z\in \Cal{D}(\kk)$, $\mbox{cl}_{\theta\kk}(\eta_{\kk}(Z)) =\overline{Z}$.\newline
(c) 
For each $Z\in \Cal{D}(\kk)$, $p\in \mbox{cl}_{\theta\kk}Z$ iff $Z\in \Cal{D}^p$.
\end{lemma}
\begin{proof}
(a) follows routinely.\newline
(b) Since $\eta_{\kk}(Z)\subseteq \overline{Z}$, therefore $\mbox{cl}_{\theta\kk}(\eta_{\kk}(Z))\subseteq\overline{Z}$. Now for the other inclusion, consider any basic 
closed set $\overline{Z_{1}}$ of $\theta\kk$, such that $\eta_{\kk}(Z)\subseteq\overline{Z_{1}}$. Then we have $\eta_{\kk}(Z) \subseteq \overline{Z_{1}}\cap \eta_\kk(\kk) =\eta_{\kk}(Z_{1})$ (from (a) above) and hence $\overline{Z}\subseteq \overline{Z_{1}}$ which means that every basic
closed set of $\theta\kk$ containing $\eta_\kk(Z)$ also contains $\overline{Z}$ and this observation leads to a conclusion that $\overline{Z}\subseteq \mbox{cl}_{\theta\kk}(\eta_{\kk}(Z))$. This proves the theorem.\newline
(c) follows routinely from (b).\newline
\end{proof}

With that in place, now we wish to characterize the maximal ideals of $\kk\la z\ra$ in terms of zero sets of functions in $\kk\la z\ra$ in the fashion of the Gelfand-Kolmogorov theorem. For $p \in \theta\kk$, denote by $M^p$ the maximal ideal corresponding to $\Cal{D}^p$. 
Also, by Lemma \ref{bijct}, and using the fact that $\theta\kk$ is in bijective correspondence with all maximal ideals of $\kk\la z\ra$, we can assert that every maximal ideal of $\kk\la z\ra$ arises this way. We have the following description for $M^p$:
\begin{theorem} \label{max}
$$M^{p}=\{f\in \kk\la z\ra : p\in \mbox{cl}_{\theta\kk}(\eta_\kk(Z_f))\},~ p\in \theta\kk.$$
\end{theorem}
\begin{proof}
We have that 
\begin{align*}
\{f\in \kk\la z\ra : p\in \mbox{cl}_{\theta\kk}(\eta_\kk(Z_f))\} & = \{f \in \kk\la z\ra ~|~ p \in \overline{Z_f}\}\text{    (by Lemma \ref{3.9}, (b))}\\
& = \{f \in \kk\la z\ra : Z_f \in \Cal{D}^p\}  = Z^{-1}[\Cal{D}^p].
\end{align*}
From Lemma \ref{bijct}, $Z^{-1}[\Cal{D}^p]$ is a maximal ideal which is  obtained in correspondence with the discrete $z$-ultrafilter $\Cal{D}^p$. This proves the result. 
\end{proof}
As a non-trivial application of Theorem \ref{max}, we show here how to obtain a chain of prime ideals of $\kk\la z\ra$ from a chain of neighbourhoods of a point in $\theta\kk \setminus \eta_\kk(\kk)$. 
\begin{theorem}\label{krull}
Given two basic open neighbourhoods $U$ and $U^{'}$ of $x \in \theta\kk \setminus \eta_\kk(\kk)$ such that $U \subsetneq U^{'}$, we can find prime ideals $P$ and $P^{'}$ respectively of $\kk\la z\ra$ such that $P, P^{'} \subseteq M^x$, and $P \subsetneq P^{'}$.
\end{theorem}
\begin{proof}
Choose $p, q \in \theta\kk \setminus \eta_\kk(\kk)$. Since $\theta\kk$ is $T_1$, there is a basic open neighbourhood $V_q$ of $q$ such that $p \notin V_q$. Define $G_{V_q} : = \{f \in \kk\la z\ra : \text{cl}_{\theta\kk}(\eta_\kk(Z_f)) \subseteq \theta\kk \setminus V_q \}$, and $N^{p, q} := M^p \bigcap M^q$. 

Now, let $\Cal{A}$ be the collection of all ideals contained in $M^p$, containing $N^{p, q}$ and disjoint from $G_{V_q}$. It can be checked that $G_{V_q} \bigcap M^p \neq \emptyset$. $\Cal{A}$ is a partially ordered set (under set inclusion), and therefore, we can find a maximal chain $\Cal{B}$ in $\Cal{A}$ by the Hausdorff maximality principle. Let $P := \bigcup \Cal{B}$. 

We will prove that $P$ is a prime ideal. If $g, h \in \kk\la z\ra \setminus P$, then the ideals $(P, g)$ (the smallest ideal containing $P$ and $\{g\}$) and $(P, h)$ must intersect $G_{V_q}$. This means that there exist $t, s \in G_{V_q}$ such that $t \equiv xg (\text{mod }P)$ and $s \equiv yh (\text{mod }P)$, where $x, y \in \kk\la z\ra$. Since $ts \neq 0 (\text{mod }P)$, $xygh \neq 0(\text{mod }P)$, which finally means that $gh \notin P$. This proves that $P$ is prime. It is clear that $P$ is strictly contained in the maximal ideal $M^p$, because $G_{V_q} \bigcap M^p \neq \emptyset$. Also, by Proposition \ref{peum}, $M^p$ is the unique maximal ideal $P$ is contained in.

Now, let $V_{q}^{'}$ be another basic open neighbourhood of $q$ such that $V_q \subsetneq V_q^{'}$ and $p \notin V^{'}_q$. Similarly defining $G_{V_q^{'}}$, we have that $G_{V_q^{'}} \subseteq G_{V_q}$. Clearly, we have $P \bigcap G_{V_q^{'}} = \emptyset$. Again, by considering the family $\Cal{A}^{'}$ as the collection of all ideals contained in $M^p$, containing $P$ and disjoint from $G_{V_q^{'}}$, we can construct another prime ideal $P^{'}$ such that $P \subseteq P^{'}$. 

We want to show that $P \subsetneqq P^{'}$. Since $V_q \subsetneqq V_q^{'}$, we have that $\theta\kk \setminus V^{'}_q \subsetneqq \theta\kk \setminus V_q$. This gives us two discrete closed sets $Z, Z^{'}$ in $\kk$, such that $ \theta\kk \setminus V_q = \text{cl}_{\theta\kk}(\eta_\kk(Z))$, $ \theta\kk \setminus V^{'}_q = \text{cl}_{\theta\kk}(\eta_\kk(Z^{'}))$, and $Z^{'} \subsetneqq Z$. Choose $g \in \kk \la z\ra$ such that $Z = Z_g$. That implies that $g \in P^{'} \setminus P$. 
\end{proof}
As an immediate corollary, we have the following:
\begin{corollary}\label{krull1}
	The Krull dimension of $\kk\la z\ra$ is at least $c$.
\end{corollary}
\begin{proof}
	Given points $p, q \in \theta\kk \setminus \eta_\kk(\kk)$, we wish to produce a chain of $c$ many basic open neighbourhoods containing $p$ and not containing $q$. 
	Since $p, q$ are arbitrary in $\theta\kk \setminus \eta_\kk(\kk)$, it suffices to demonstrate a chain containing $c$ many discrete closed subsets of $\mathbb{R}^2$. 
	
	Take the set $\ZZ \times \{0\}$, and bijectively map it to $\mathbb{Q}$. For an irrational number $r$, $\mathbb{Q}_r := \{ x \in \mathbb{Q}: x < r\}$ is a chain of subsets in $\mathbb{Q}$. Taking the inverse images of $\mathbb{Q}_r$ in $\ZZ \times \{0\}$ and varying $r$ will give a chain of discrete closed subsets of $\ZZ\times\{0\} \subset \RR^2$ containing $c$ many elements.
	\end{proof}
Observe that Corollary \ref{krull1} would extend to give another proof of the fact that Krull dimension of $\kk\langle z\rangle$ is $2^{c}$ if one could establish the existence of a neighbourhood chain around a point $x \in \theta\kk \setminus \eta_\kk(\kk)$ containing $2^c$ elements. It might be an interesting question to investigate whether such a chain exists.

Now, we go for another important claim: 
\begin{theorem}\label{ska}
$\theta\kk$ is a Wallman compactification of $\eta_{\kk}(\kk)$.
\end{theorem}
\begin{proof}
It is sufficient to show that the collection of all closed subsets of $\eta_{\kk}(\kk)$ is none other than the $\eta_{\kk}$-image of closed discrete subsets of $\kk$, because from the construction of $\theta\kk$ it follows immediately. Now basic closed sets of $\eta_{\kk}(\kk)$ are of the form $\overline{Z}\cap \eta_{\kk}(\kk)=\eta_{\kk}(Z)$, where $\overline{Z}$ is a basic closed set of $\theta\kk$. 
Then $A=\{\eta_{\kk}(Z) : Z\in \Cal{D}(\kk)\}$ forms a base for the closed sets of $\eta_{\kk}(\kk)$. Since under arbitrary intersection every discrete set remains discrete, therefore, $A$ itself is the collection of all closed subsets of $\eta_{\kk}(\kk)$.
\end{proof}

We end this section by proving that $\eta_\RR(\RR)$ and $\eta_\CC(\CC)$ are homeomorphic. This automatically implies that such a homeomorphism would readily extend to a homeomorphism between their Wallman compactifications $\theta\RR$ and $\theta\CC$. However, we also include an explicit demonstration of this fact. 
\begin{proposition}\label{churi}
(a) $\eta_\RR(\RR)$ and $\eta_\CC(\CC)$ are homeomorphic.\newline
(b) $\theta\RR$ and $\theta\CC$ are homeomorphic.
\end{proposition}
\begin{proof}
(a) Observe that the basic closed sets of $\eta_\kk(\kk)$ are given by the closed discrete sets of $\kk$ in the usual topology. So, it suffices to design a bijection from $\RR$ to $\CC$ which takes discrete closed sets to discrete closed sets (in the usual topology). Using sequential compactness of $\kk$, it suffices to define a bijective map from $\RR$ to $\CC$ that takes compact sets to compact sets. For example, start by defining a bijective map $f = (f_1, f_2) : \ZZ \to \ZZ \times \ZZ$. Now, extend $f$ to $F : \RR \to \CC$ which maps $[n, n + 1)$ to  $[f_1(n), f_1(n) + 1) \times [f_2(n), f_2(n) + 1)$ bijectively. This extended map $F$ gives a homeomorphism.

(b) Pick a discrete $z$-ultrafilter $\Cal{F}$ from $\theta\RR$. Using the map $F$ from (a), it is clear that $F(\Cal{F}) = \{F(S) : S \in \Cal{F}\}$ is a discrete $z$-ultrafilter on $\CC$ and hence lies in $\theta\CC$. Also, this mapping between discrete $z$-ultrafilters is bijective, which follows from the bijectivity of $F$. Lastly, a basic closed set of $\theta\RR$ is given by $\overline{Z} = \{\mathfrak{F} \in \theta\RR : Z \in \mathfrak{F}\}$. Then, we have that $F(\overline{Z}) = \{\mathfrak{G} \in \theta\CC : F(Z) \in \mathfrak{G}\}$, which is a basic closed set in $\theta\CC$. 
\end{proof}

\subsection{Acknowledgements}
The authors would like to thank Sudip Kumar Acharyya for reading through a draft copy of this paper and providing important suggestions. The second author is thankful to MPIM Bonn for its financial support and for providing ideal working conditions. 

\section{Appendix: Assorted properties of $\beta\kk$, $\theta\kk$ and $\Psi$} \label{app}
We use this appendix to record some assorted facts regarding $\theta\kk$, $\eta_\kk{\kk}$ and the map $\Psi$. These results seem to us to be of independent interest and some might also spur further investigation. We start by outlining one further connection between Stone-\v{C}ech compactification and $\theta\KK$, namely
\begin{proposition}
	For $Z \in \Cal{D}(\kk)$, $\text{cl}_{\theta\kk}(\eta_{\kk}(Z))$ is homeomorphic to $\beta Z$.
\end{proposition}
\begin{proof}
	The non-trivial case is when $Z$ is infinite. By definition, $\beta Z$ consists of all $z$-ultrafilters of zero-sets of $Z$. For $p \in \text{cl}_{\theta\kk}(\eta_{\kk}(Z))$, let $A^p : = \{Z_1 \in \Cal{D}^p : Z_1 \subseteq Z\}.$ It is clear that $A^p$ is a $z$-ultrafilter on $Z$.
	
	Now, consider the map $\Phi_1 : \text{cl}_{\theta\kk}(\eta_{\kk}(Z)) \mapsto \beta Z$, given by 
	\[
	\Phi_1 (\Cal{D}^p) = A^p.
	\]
	Clearly, $\Phi_1$ is bijective. Again, any basic closed subset of $\text{cl}_{\theta\kk}(\eta_{\kk}(Z))$ is of the form $\text{cl}_{\theta\kk}(\eta_{\kk}(Z_1))$, where $Z_1 \subseteq Z$. That means, 
	\begin{align*}
	\Phi_1 (\text{cl}_{\theta\kk}(\eta_{\kk}(Z_1))) & = \{ A^p : p \in \text{cl}_{\theta\kk}(\eta_{\kk}(Z_1))\}\\
	& = \{A^p : Z_1 \in A^p\} = \text{cl}_{\beta Z} Z_1,
	\end{align*}
	which is a basic closed set in $\beta Z$, making $\Phi_1$ a closed map. One can similarly check that $\Phi_1$ is continuous, proving the claim.
\end{proof}

We know from the property of Stone-\v{C}ech compactification that for a first countable Tychonoff topological space $X$ no point of $\beta X\setminus X$ is $G_{\delta}$ in $\beta X$ (see ~\cite{GJ}, Chapter 9). It follows that no point of $\beta\mathbb{K}\setminus \mathbb{K}$ is $G_{\delta}$ in $\beta \mathbb{K}$. Now, we prove that no point of $(\beta\mathbb{K}\setminus Q(\mathbb{K}))\setminus \mathbb{K}$ is $G_{\delta}$ in $\beta\mathbb{K}\setminus Q(\mathbb{K})$.

For the following proposition, for notational convenience, let 
$\kk^* = \beta \kk \setminus \kk, \kk^*_Q = (\beta\kk \setminus  Q(\kk)) \setminus \kk, \beta_Q(\kk) = \beta\kk \setminus Q(\kk)$.
Then, we have the following proposition, which is a variant of Lemma 9.4, Chapter 9 of ~\cite{GJ}:
\begin{proposition}
	Let $E \subset \beta_Q\kk$, and suppose that $Z$ is a zero set in $\beta_Q\kk$ that meets $cl_{\beta_Q\kk}E$ but not $\kk \cup E$. Then $E$ contains a copy $N$ of $\NN$ and $Z$ contains a copy of $\beta N \setminus N$.
	\end{proposition}
	\begin{proof}
		Let $Z = Z_f$, and $Y = \beta_Q\kk \setminus Z$. Since $Z \cap \kk = \emptyset$, we have $Y \supset \kk$, and therefore $\beta Y = \beta \kk$. Also, $Y \supset E$. In $C(Y)$, $h = (f|_Y)^{-1}$ exists, and because $Z$ meets $cl_{\beta_Q\kk}E$, $h$ is unbounded on $E$. Then $E$ contains a copy $N$ of $\NN$ that is $\kk$-embedded in $Y$ and on which $h$ goes to infinity. This means that $N$ is $\kk^*$-embedded in $\beta Y = \beta \kk$, which in turn means that $cl_{\beta(\beta_Q\kk)} N = \beta N\Rightarrow cl_{\beta \kk} N = \beta N $. Therefore $\beta N$ contains no far points of $\beta \kk$, as it is the closure of a discrete subset $N$ of $\kk$. Hence $\beta N \subseteq \beta\kk \setminus  Q(\kk)$.
		
		Again, $cl_{\beta_Q\kk} N \setminus N \subset Z \implies \beta N \subset Z$.
		\end{proof}
		\begin{remark}\label{3.14}\hfill
			\begin{enumerate}
				\item If we assume $E = \kk$, then it tells us that every non-empty zero set in $\kk^*_Q$, if disjoint from $\kk$, contains a copy of $\beta \NN$ and hence its cardinality is at least $2^c$, where $c$, as usual denotes the cardinality of $\RR$.
				\item Since every compact $G_\delta$-set contains a zero set in a completely regular space, therefore no point of $\kk_Q^*$ is a $G_\delta$-point of $\beta_Q\kk$. 
				\end{enumerate}
				\end{remark}

				It is clear that no point $p$ of $\eta_\kk(\kk)$ is a $G_\delta$ point, because otherwise, $\eta_\kk(\kk) \setminus \{p\}$ is an uncountable set which is a countable union of closed sets in $\eta_\kk(\kk)$, which cannot happen. More non-trivial is the following
				\begin{proposition}\label{sesh}
					No point of $\theta\kk \setminus \eta_\kk(\kk)$ is a $G_\delta$-point of $\theta\kk$.
					\end{proposition}
					\begin{proof} The above claim follows from contradiction. Let there be $p \in \theta\kk \setminus \kk$ such that $p$ is a $G_\delta$-point of $\theta\kk$. Then there exists a countable collection of open neighbourhoods $V_i$ of $p$ in $\theta\kk$ such that $\bigcap^\infty_{i = 1} V_i = \{p\}.$ Since $\Psi$ is bijective and continuous, we have
					\[
						\Psi^{-1}(\bigcap_{i = 1}^\infty V_i) = \Psi^{-1}(\{p\}) \implies \bigcap_ {i = 1}^\infty\Psi^{-1}(V_i) = \Psi^{-1}(p),
						\]
						which means that $\Psi^{-1}(p)$ is a $G_\delta$-point of $\beta_Q\kk$. But no point of $\beta_Q\kk\setminus\kk$ is a $G_\delta$-point of $\beta_Q\kk$ (Remark \ref{3.14} above), which means that $\Psi^{-1}(p) \in \kk$. This is not possible as $\Psi$ maps $\kk \subset \beta_Q\kk$ to $\eta_\kk(\kk)\subset \theta\kk$, and therefore, it maps $\beta_Q\kk\setminus\kk$ to $\theta\kk\setminus \eta_\kk(\kk)$.
						\end{proof}  

\bibliographystyle{plain}
\def\noopsort#1{}

\end{document}